\documentclass[10pt,reqno]{amsart}

\usepackage{enumerate}
\usepackage{amsmath}
\usepackage{amssymb}
\usepackage{amsfonts}
\usepackage{amsthm}
\usepackage{color}

\newtheorem{thm}{Theorem}[section]

\newtheorem{prop}[thm]{Proposition}
\newtheorem{cor}[thm]{Corollary}

\theoremstyle{definition}
\newtheorem{dfn}[thm]{Definition}
\newtheorem{example}[thm]{Example}

\theoremstyle{remark}

\def\Correspondingauthor{$^{*}$\protect\footnotetext{$^{*}$ C\lowercase{orresponding author.}}}

\numberwithin{equation}{section}

\begin{document}

\title[Generalized Bessel multipliers]%
{Generalized Bessel multipliers in Hilbert spaces}

\author[Abbaspour, Hossein-nezhad, Rahimi]{GH. ABBASPOUR TABADKAN\Correspondingauthor, H. Hossein-nezhad and A. Rahimi}

\address{School of Mathematics and Computer Sciences,
Damghan University, Damghan, Iran.}
\email{abbaspour@du.ac.ir; gabbaspour@yahoo.com }
\email{h.hosseinnezhad@std.du.ac.ir; hosseinnezhad\_h@yahoo.com}
\address{Department of Mathematics, University of Maragheh, P. O. Box 55136-553, Maragheh,Iran.}
\email{rahimi@maragheh.ac.ir}


\date{\today}

\begin{abstract}
The notation of generalized Bessel multipliers is obtained by a bounded operator on $\ell^2$
which is inserted between the analysis and synthesis operators.
We show that various properties of generalized multipliers
are closely related to their parameters, in particular it will be shown that the membership of
generalized Bessel multiplier in the certain operator classes requires that its symbol belongs in the same classes, in special sense.
Also, we give some examples to illustrate our results.
As we shall see, generalized multipliers associated with Riesz bases are well-behaved, more precisely
in this case multipliers can be easily composed and inverted.
Special attention is devoted to the study of invertible generalized multipliers.
Sufficient and/or necessary conditions for invertibility are determined.
Finally, the behavior of these operators under
perturbations is discussed.
\end{abstract}

\maketitle

\section{Introduction}
\label{sec:intro}
Frames for Hilbert spaces were introduced in 1952 by Duffin and Schaefer \cite{R.J. Duffin and A.C. Schaeffer-1952}
in the context of nonharmonic Fourier series.
Later, Daubechies,
Grossmann and Meyer \cite{Meyer-1986} found a new fundamental application to wavelet and Gabor
transforms in which frames played an important role.
Moreover, frames are main tools for signal and image processing \cite{Bolcskei-1998, Vetterli-1995, Shen-2006}, data compression \cite{Casazza-2003}, sampling theory \cite{Aldroubi-2001}, optics \cite[Ch.14]{Feichtinger and Strohmer-1998}, signal detection \cite[Ch.11]{Feichtinger and Strohmer-1998}, filter banks \cite{Benedetto-1998}, etc.

Several notions generalizing the concept of frames
have been introduced and studied, namely; Banach frames, pseudo frames,
fusion frames (or frames of subspaces), G -frames and etc. See for example
\cite{Casazza-Larson-1999,
Li-Ogawa-2004,
Casazza-Kutyniok-Li-2008,
Wenchang-2006}.

Although frames are useful tools in applications and theory, there are many systems that do not satisfy both frame conditions
at the same time, therefore the concepts semi frames \cite{JAntoine-Balazs} and reproducing pairs \cite{Speckbacher-Balazs, JAntoine} have also been recently introduced.

Bessel multipliers in Hilbert spaces were introduced by Balazs in \cite{Balazs-2007}.
Bessel multipliers are operators that are defined by a fixed multiplication pattern which is inserted between the analysis and synthesis operators.
We refer to \cite{Balazs-2007} for an introduction to the concept of Bessel multipliers and theirs properties.
Reproducing pairs are closely related to frame multipliers \cite{Balazs-2007, Balazs-Rahimi. Invertibility is
a central topic in the mathematical investigation of multipliers.
This includes the question under which conditions a system of two sequences
forms a reproducing pair.}

The standard matrix representation of operators using an orthonormal basis was presented in \cite{J.B. Conway} and
it has been generalized in several directions. One of the recent directions of such generalizations is investigated by Balazs \cite{Balazs-2008}
by using the Bessel sequences, frames and Riesz bases. In the same paper, the author also established the function which assigns an
operator in $\mathcal{B}(\mathcal{H}_1,\mathcal{H}_2)$ to an infinite matrix in $\mathcal{B}(\ell^2)$. The last concept is a generalization of Bessel multiplier as introduced in \cite{Balazs-2009}.

In this paper, in addition to recalling some results from \cite{Balazs-2008, Balazs-2009}, we also derive new original results.
Particularly, we interested to study the possibility of invertibility of generalized multipliers depending on the properties of its
corresponding sequences and its symbol. Finally,
we investigate the behavior of generalized Bessel multipliers when the parameters are changing.

\section{Notation and preliminaries}\label{sec2}
Throughout the paper, $\mathcal{H}$ will denote a separable Hilbert space and $\mathcal{B}(\mathcal{H}_1,\mathcal{H}_2)$
is the Banach space of all bounded linear operators from $\mathcal{H}_1$ to $\mathcal{H}_2$ with operator norm.
At the first, we recall some definitions.\\
A countable family of elements $\{f_k\}_{k=1}^{\infty}$ in $\mathcal{H}$ is a
\begin{enumerate}
\item \emph{Bessel sequence} if there exists a constant $B>0$ such that
   \begin{equation}\label{eqn:Bessel intro}
   \sum_{k=1}^{\infty}|\langle f,f_{k}\rangle|^2\leq B\| f\|^2,\hspace{4mm}  (f\in \mathcal{H}),
   \end{equation}
\item \emph{frame} for $\mathcal{H}$ if there exist constants $A, B>0$ such that
    \begin{equation}\label{eqn:frame intro}
  A\| f\|^2\leq \sum_{k=1}^{\infty}|\langle f,f_{k}\rangle|^2\leq B\| f\|^2,\hspace{4mm} (f\in \mathcal{H}),
    \end{equation}
the numbers $A$, $B$ in  \eqref{eqn:frame intro} are called frame bounds,
  \item \emph{Riesz basis} for $\mathcal{H}$ if $\overline{span}\{f_k\}_{k=1}^{\infty}$ = $\mathcal{H}$ and there exist constants $A,B>0$ such that
   \begin{equation}\label{eqn:Riesz intro}
   A \sum_{k=1}^{\infty}|c_{k}|^{2}\leq\|\sum_{k=1}^{\infty}c_{k}f_{k}\|^{2}\leq B\sum_{k=1}^{\infty}|c_{k}|^{2},
   \end{equation}
   for every finite scalar $\{c_{k}\}\in \ell^2$.
\end{enumerate}
Every orthonormal basis is a Riesz basis, and every Riesz basis is a frame (the bounds coincide).
The difference between a Riesz basis and a frame is that the elements in a frame might be dependent. More precisely, a frame $\{f_k\}_{k=1}^{\infty}$
is a Riesz basis if and only if
\begin{equation}\label{eqn:difference between a Riesz basis and a frame}
  \sum_{k} c_{k}f_{k}=0,\hspace{2mm}\{c_{k}\}\in \ell^2(\mathbb{N})\Rightarrow c_{k}=0 \hspace{3mm} (k\in \mathbb{N}).
\end{equation}
A frame that is not a Riesz basis is said to be \emph{overcomplete}.
The option of having overcompleteness in a frame makes
the concept more flexible than that of a basis: we have more freedom, which
enhances the chance that we can construct systems having prescribed properties.
The overcompleteness is also useful in practice, e.g., in the context of signal transmission. More details can be found in \cite{Christensen-2016}.



We denote the synthesis operator and the analysis operator associated to a Bessel sequence $\{f_k\}_{k=1}^{\infty}$ by $\mathcal{D}_{\{f_{k}\}}$ and $\mathcal{C}_{\{f_{k}\}}$, respectively, which are
defined as follow
\begin{equation}\label{eqn:syn op}
\mathcal{D}_{\{f_{k}\}}:\ell^2(\mathbb{N})\rightarrow \mathcal{H},\hspace{3mm} \mathcal{D}_{\{f_{k}\}}(\{c_k\})=\sum_{k}c_{k}f_{k},
\end{equation}
\begin{equation}\label{eqn:anal op}
\mathcal{C}_{\{f_{k}\}}:\mathcal{H} \rightarrow \ell^2(\mathbb{N}),\hspace{3mm} \mathcal{C}_{\{f_{k}\}}(f)= \{\langle f,f_{k}\rangle\}_{k}.
\end{equation}
Composing $\mathcal{D}_{\{f_{k}\}}$ and $\mathcal{C}_{\{f_{k}\}}$, we obtain the (associated) frame operator
\begin{equation}\label{eqn:frame op}
 \mathcal{S}_{\{f_{k}\}}:\mathcal{H} \rightarrow \mathcal{H},\hspace{3mm} \mathcal{S}_{\{f_{k}\}}(f)= \sum_{k}\langle f,f_{k}\rangle f_{k}.
\end{equation}

We state some of the important properties of the mentioned operators; proofs can be found in \cite{Christensen-2016}.
The index set will be omitted in the following, if no distinction is necessary.
\begin{prop}[\cite{Christensen-2016}]\label{Bessel bounds}
Let $\{f_k\}_{k=1}^{\infty}$ be a Bessel sequence for $\mathcal{H}$. Then $\| f_{k}\|\leq \sqrt{B}$ and the operators
$\mathcal{D}$ and $\mathcal{C}$ are adjoint to each other. Moreover $\|\mathcal{D}\|_{op} =\|\mathcal{C}\|_{op}\leq\sqrt{B}$.
\end{prop}
\begin{thm}[\cite{Christensen-2016}]
Let $\{f_k\}_{k=1}^{\infty}$ be a frame with frame operator $\mathcal{S}$ and frame bounds $A$ and $B$. Then the following hold:
\begin{itemize}
  \item $\mathcal{S}$ is bounded, invertible, self-adjoint and positive.
  \item $\{\mathcal{S}^{-1}f_k\}_{k=1}^{\infty}$ is a frame with frame operator $\mathcal{S}^{-1}$ and frame bounds $B^{-1}$ and $A^{-1}$. Also every $f\in\mathcal{H}$
  has expansions $f=\sum_{k}\langle f,\mathcal{S}^{-1}f_{k}\rangle f_{k}$ and $f=\sum_{k}\langle f,f_{k}\rangle \mathcal{S}^{-1}f_k$, where both sums converge unconditionally in $\mathcal{H}$.
\end{itemize}
\end{thm}
The frame $\{\mathcal{S}^{-1}f_k\}_{k=1}^{\infty}$ is called the \emph{canonical dual frame} of $\{f_k\}_{k=1}^{\infty}$.The reason for the name is that it
plays the same role in frame theory as the dual basis in the theory of bases. We denote the canonical dual frame of $\{f_k\}_{k=1}^{\infty}$ by $\{\tilde{f_{k}}\}_{k=1}^{\infty}$.

If $\{f_{k}\}_{k=1}^{\infty}$ is an overcomplete frame, then by \cite[Lemma 6.3.1]{Christensen-2016} there exist frames $\{g_{k}\}_{k=1}^{\infty}\neq\{\tilde{f}_{k}\}_{k=1}^{\infty}$
for which
\begin{equation}\label{dual property}
f=\sum_{k}\langle f,g_{k}\rangle f_{k},~~(f\in \mathcal{H}).
\end{equation}
$\{g_{k}\}_{k=1}^{\infty}$ is called an alternative dual of $\{f_{k}\}_{k=1}^{\infty}$.

In \cite{Schatten-1960}, R. Schatten provided a detailed study of ideals of compact operators using their singular decomposition. He
investigated the operators of the form $\sum_{k}\lambda_{k}g_{k}\otimes f_{k}$, where $\{f_{k}\}_{k=1}^{\infty}$ and $\{g_{k}\}_{k=1}^{\infty}$ are orthonormal families.
In \cite{Balazs-2007}, the orthonormal families were replaced with Bessel and frame sequences to define Bessel and frame multipliers.

Let $\mathcal{H}_{1}$ and $\mathcal{H}_{2}$ be Hilbert spaces. Let $\{f_{k}\}_{k=1}^{\infty}\subseteq\mathcal{H}_{1}$ and
$\{g_{k}\}_{k=1}^{\infty}\subseteq\mathcal{H}_{2}$ be Bessel sequences. Fix  $m=\{m_{k}\}_{k=1}^{\infty}\in\ell^{\infty}$.
The operator
\begin{equation}\label{eqn:multi intro}
\mathbf{M}_{m,\{g_{k}\},\{f_{k}\}}: \mathcal{H}_{1}\rightarrow\mathcal{H}_{2},\hspace{3mm}\mathbf{M}_{m,\{g_{k}\},\{f_{k}\}}(f)=\sum_{k}m_{k}\langle f,f_{k}\rangle g_{k}
\end{equation}
is called the \emph{Bessel multiplier} of the Bessel sequences $\{f_{k}\}_{k=1}^{\infty}$ and $\{g_{k}\}_{k=1}^{\infty}$.
The sequence $m$ is called the symbol of $\mathbf{M}$.

Also, the Bessel multiplier can be considered in terms of the synthesis and analysis operators. Let the mapping $\mathcal{M}_m:\ell^{2}\rightarrow\ell^{2}(m\in\ell^{p})$ is given by the pointwise multiplication $\mathcal{M}_{m}(\{c_k\})=\{m_{k}c_{k}\}$.
So a Bessel multiplier $\mathbf{M}_{m}$ can be written as $\mathbf{M}_{m}= {\mathcal{D}}_{\{g_{k}\}}o\mathcal{M}_{m}o~{\mathcal{C}}_{\{f_{k}\}}$.


Before mentioning some properties of Bessel multipliers, it will be convenient to recall some basic facts regarding the operators on Hilbert spaces.

If $f,g$ are elements of a Hilbert space $\mathcal{H}$, we define the \emph{tensor operator} $f\otimes g$ on $\mathcal{H}$
by $(f\otimes g)(h) = \langle h, g\rangle f$, for $h\in\mathcal{H}$. For an operator $\mathcal{U}\in \mathcal{B}(\mathcal{H})$,
the following equalities are readily verified:
\begin{equation*}
\mathcal{U}(f\otimes g) = \mathcal{U}(f)\otimes g,\hspace{5mm}(f\otimes g)\mathcal{U} = f\otimes \mathcal{U}^{*}(g).
\end{equation*}
Each finite rank operator is a linear combination of rank one operators of the form $f\otimes g$, for $f, g\in\mathcal{H}$.

If $\mathcal{U}$ is a compact operator, then there exist orthonormal sets $\{e_k\}^{\infty}_{k=1}$ and $\{\sigma_k\}^{\infty}_{k=1}$ in $\mathcal{H}$ such that
$$\mathcal{U}f = \sum_k \lambda_k\langle f, e_k\rangle \sigma_k,\hspace{3mm}(f\in\mathcal{H}),$$
where $\lambda_k$ is the k-th singular value of $\mathcal{U}$, \cite{zhuo}. Given
$0<p<+\infty$, we define the \emph{Schatten p-class} of $\mathcal{H}$ \cite{Schatten-1960}, denoted $\mathcal{S}_p(\mathcal{H})$, to be space of all compact operators $\mathcal{U}$ on $\mathcal{H}$ with its singular value sequence $\{\lambda_k\}$ belonging to $\ell^p$. It is known that $\mathcal{S}_p(\mathcal{H})$ is a Banach space with the norm
$$\|\mathcal{U}\|_p = \Big[\sum_k |\lambda_k|^p\Big]^{\frac{1}{p}}.$$
$\mathcal{S}_1(\mathcal{H})$ is also called \emph{trace-class} and $\mathcal{S}_2(\mathcal{H})$ is usually called the \emph{Hilbert-Schmidt} class.
If $\mathcal{U}$ is a trace-class operator on $\mathcal{H}$, then the trace-class norm can also be calculated by
$$\|\mathcal{U}\|_{1} := \sum_k \langle |\mathcal{U}|e_k, e_k\rangle,$$
where $|\mathcal{U}|$ is the operator for which $|\mathcal{U}| = (\mathcal{U}^{*}\mathcal{U})^{1/2}$ and $\{e_k\}_{k=1}^{\infty}$ is any orthonormal basis of $\mathcal{H}$ \cite{Schatten-1960}.
Moreover, the compact operator $\mathcal{U}$ is Hilbert-Schmidt if and only if $\|\mathcal{U}\|_2:=\sum_k \|\mathcal{U}e_k\|^2<+\infty$, for all orthonormal bases of $\mathcal{H}$ \cite{zhuo}.

It is proved that $\mathcal{S}_p(\mathcal{H})$ is a two sided $*$-ideal of $\mathcal{B}(\mathcal{H})$. Analogously, for Hilbert spaces $\mathcal{H}_1, \mathcal{H}_2, \mathcal{H}_3, \mathcal{H}_4$ and for operators $\mathcal{U}\in \mathcal{B}(\mathcal{H}_1,\mathcal{H}_2), \mathcal{V}\in\mathcal{B}(\mathcal{H}_4,\mathcal{H}_3)$ and
$\mathcal{W}\in\mathcal{S}_p(\mathcal{H}_3,\mathcal{H}_1)$, we have
$\mathcal{UW}\in\mathcal{S}_p(\mathcal{H}_3,\mathcal{H}_2)$
and $\mathcal{WV}\in\mathcal{S}_p(\mathcal{H}_4,\mathcal{H}_1)$.
Moreover, if $\mathcal{U}\in\mathcal{S}_p$ and $\mathcal{V}\in\mathcal{B}(\mathcal{H})$, then $\|\mathcal{UV}\|_p\leq\|\mathcal{U}\|_p\|\mathcal{V}\|$ and $\|\mathcal{VU}\|_p\leq\|\mathcal{U}\|_p\|\mathcal{V}\|$. We refer to \cite{Schatten-1960, zhuo, Murphy} for more detailed information about these operators.
\begin{thm}[\cite{Balazs-2007}]\label{balazs main theorem}
Let $\mathbf{M}=\mathbf{M}_{m,\{g_{k}\},\{f_{k}\}}$ be a Bessel multiplier for the Bessel sequences $\{f_{k}\}_{k=1}^{\infty}\subseteq\mathcal{H}_{1}$ and
$\{g_{k}\}_{k=1}^{\infty}\subseteq\mathcal{H}_{2}$ with the bounds $B$ and $B'$, respectively. Then
\begin{itemize}
  \item If $m\in\ell^{\infty}$, $\mathbf{M}$ is a well-defined bounded operator with
 \mbox{$\|\mathbf{M}\|_{op}\leq\sqrt{BB'}\| m\|_{\infty}$}. Moreover the sum $\sum_{k}m_{k}\langle f,f_{k}\rangle g_{k}$
converges unconditionally for all $f\in\mathcal{H}_{1}$.
  \item ($\mathbf{M}_{m,\{g_{k}\},\{f_{k}\}})^{*} = \mathbf{M}_{\bar{m},\{f_{k}\},\{g_{k}\}}$.
  \item If $m\in c_{0}$, $\mathbf{M}$ is a compact operator.
  \item If $m\in\ell^{1}$, $\mathbf{M}$ is a trace-class operator with  $\|\mathbf{M}\|_{1}\leq\sqrt{BB'}\| m\|_{1}$.
  \item If $m\in\ell^{2}$, $\mathbf{M}$ is a  Hilbert-Schmidt operator with $\|\mathbf{M}\|_{2}\leq\sqrt{BB'}\| m\parallel_{2}$.
\end{itemize}
\end{thm}
\section{Generalized Bessel multipliers}
It is well known that each linear operator is represented by corresponding matrix, and the
opposite is also true: each matrix generates corresponding linear operator. Some authors have constructed various
types of matrix representation of operators using an orthonormal basis \cite{J.B. Conway}, frames and their canonical duals \cite{Christensen-1995},
Gabor frames \cite{groching-2006},  Bessel sequences \cite{Balazs-2008} and localized frames \cite{Balazs-2016}.
The operators induced by matrices, with respect to Bessel sequences, are also introduced by Balazs \cite{Balazs-2008}.

In the sequel we focus on the properties of these operators.
\begin{dfn}\label{bessel u multi}
Let $\mathcal{H}_{1}$ and $\mathcal{H}_{2}$ be Hilbert spaces and $\{f_{k}\}_{k=1}^{\infty}\subseteq\mathcal{H}_{1}$ and
$\{g_{k}\}_{k=1}^{\infty}\subseteq\mathcal{H}_{2}$ be Bessel sequences. Also suppose that $\mathcal{U}$ is an arbitrary
non zero bounded and
linear operator on $\ell^2$, which can be considered as a matrix. Then, the operator $\mathbf{M}_{\mathcal{U},\{g_{k}\},\{f_{k}\}}:\mathcal{H}_{1}\rightarrow\mathcal{H}_{2}$ defined by
\begin{equation}\label{eqn:g-multi intro}
\mathbf{M}_{\mathcal{U},\{g_{k}\},\{f_{k}\}}=\mathcal{D}_{\{g_{k}\}}o\hspace{1mm}\mathcal{U}\hspace{1mm}o\hspace{1mm}\mathcal{C}_{\{f_{k}\}}.
\end{equation}
is called the \emph{generalized Bessel multiplier}.
\end{dfn}
The operator $\mathcal{U}$ is called the symbol of $\mathbf{M}$. In particular if $\mathcal{U}=\mathcal{M}_{m}$, the mentioned pointwise multiplication, then
$\mathbf{M}_{\mathcal{U}}$ is a Bessel multiplier. Clearly, the roles of sequences and symbol in \eqref{bessel u multi} is very important.

If we denote $\mathbf{M}_{\mathcal{U},\{f_k\}}=\mathbf{M}_{\mathcal{U},\{f_k\},\{f_k\}}$, then it is easy to check that for
orthonormal sequence $\{e_{k}\}_{k=1}^{\infty}$, we have the '\textbf{symbolic calculus}' as follows
\begin{equation}\label{symbolic calculus}
\mathbf{M}_{\mathcal{U}_{1},\{e_{k}\}} o ~\mathbf{M}_{\mathcal{U}_{2},\{e_{k}\}} = \mathbf{M}_{\mathcal{U}_{1} o ~\mathcal{U}_{2},\{e_{k}\}}.
\end{equation}
Now, we plan to establish analogues of Theorem \ref{balazs main theorem} for generalized Bessel multipliers.
It is worth mentioning that some of known results in the following theorem have already been stated in \cite{Balazs-2008, Balazs-2009}
without proofs. However, we include the proofs for the sake of completeness.


The interested reader can find the properties of these operators in \cite{Murphy}.
\begin{thm}\cite[Theorem 6.0.4]{Balazs-2009}\label{main theorem}
Let $\mathbf{M}=\mathbf{M}_{\mathcal{U},\{g_k\},\{f_k\}}$ be the generalized Bessel multiplier for the Bessel sequences $\{f_{k}\}_{k=1}^{\infty}\subseteq\mathcal{H}_{1}$ and $\{g_{k}\}_{k=1}^{\infty}\subseteq\mathcal{H}_{2}$ with the bounds $B$ and $B'$, respectively. Then
\begin{enumerate}
  \item If $\mathcal{U}\in\mathcal{B}(\ell^2)$, then $\mathbf{M}$ is a well-defined and bounded operator with
  $\|\mathbf{M}\|_{op}\leq\sqrt{BB'}\|\mathcal{U}\|_{op}$.
  \item $(\mathbf{M}_{\mathcal{U},\{g_k\},\{f_k\}})^{*}=\mathbf{M}_{\mathcal{U}^{*},\{f_k\},\{g_k\}}$.
  \item If $\mathcal{U}$ is a compact operator on $\ell^2$, then $\mathbf{M}$ is a compact operator from $\mathcal{H}_1$ to $\mathcal{H}_2$.
\item If $\mathcal{U}\in\mathcal{S}_p(\mathcal{H}_1, \mathcal{H}_2)$, then $\mathbf{M}\in\mathcal{S}_p(\mathcal{H}_1, \mathcal{H}_2)$ with $\|\mathbf{M}\|_p\leq\sqrt{BB'}\|\mathcal{U}\|_p$.
  \item If $\mathcal{U}\in\mathcal{B}(\ell^2)$ is a positive operator, then $\mathbf{M}_{\mathcal{U},\{f_k\}}\in\mathcal{B}(\mathcal{H}_1)$ is also positive.
  \end{enumerate}
\end{thm}
\begin{proof}
(1) Since $\{g_{k}\}_{k=1}^{\infty}$ is a Bessel sequence and for each $f\in \mathcal{H}_{1}$, we have $\mathcal{U}o~\mathcal{C}_{\{f_{k}\}}\in\ell^2$, so \eqref{eqn:g-multi intro} is a well-defined operator. Furthermore
\begin{equation*}
\|\mathbf{M}\|_{op}=\|\mathcal{D}_{\{g_{k}\}}o\hspace{1mm}\mathcal{U}\hspace{1mm}o\hspace{1mm}\mathcal{C}_{\{f_{k}\}}\|_{op}
\leq\|\mathcal{D}_{\{g_{k}\}}\|_{op}.
\|\mathcal{U}\|_{op}.\|\mathcal{C}_{\{f_{k}\}}\|_{op}\leq\sqrt{BB'}\|\mathcal{U}\|_{op}.
\end{equation*}
(2) Clearly by Definition \ref{eqn:g-multi intro}, for each $f\in\mathcal{H}_1$ and $g\in\mathcal{H}_2$, we have:
\begin{equation*}
\begin{split}
\mathbf{M}^{*}_{\mathcal{U},\{g_k\},\{f_k\}}
& =(\mathcal{D}_{\{g_k\}}o~\mathcal{U}o~\mathcal{C}_{\{f_k\}})^{*}\\
& =(\mathcal{C}^{*}_{\{f_k\}}o~\mathcal{U}^{*}o~\mathcal{D}^{*}_{\{g_k\}})\\
& =(\mathcal{D}_{\{f_k\}}o~\mathcal{U}^{*}o~\mathcal{C}_{\{g_k\}})\\
& =\mathbf{M}_{\mathcal{U}^{*},\{f_k\},\{g_k\}}
\end{split}
\end{equation*}
(3) First we show that $\mathbf{M}$ is a finite rank operator if $\mathcal{U}$ is one. Let $\mathcal{U}$ be a finite rank operator on $\ell^2$. Without loss of generality, we may assume that $\mathcal{U}$ is positive. Then $\mathcal{U}=\sum_{j=1}^{n}(x_j\otimes x_j)$,
for some vectors $\{x_{j}\}\in\ell^2$. Hence

\begin{equation*}
\begin{split}
\mathbf{M}=\mathcal{D}_{\{g_{k}\}}o\hspace{1mm}\mathcal{U}\hspace{1mm}o\hspace{1mm}\mathcal{C}_{\{f_{k}\}}
& =\sum_{j=1}^{n}\Big(\mathcal{D}_{\{g_{k}\}}(x_j)\otimes\mathcal{C}_{\{f_{k}\}}^{*}(x_j)\Big)\\
& = \sum_{j=1}^{n}\Big(\mathcal{D}_{\{g_{k}\}}(x_j)\otimes\mathcal{D}_{\{f_{k}\}}(x_j)\Big),\\
\end{split}
\end{equation*}
since the operators $\mathcal{D}_{\{g_{k}\}}(x_j)\otimes\mathcal{D}_{\{f_{k}\}}(x_j)$ are finite rank, we are done.

Now, suppose $\mathcal{U}$ is a compact operator on $\ell^2$. Then, by \cite[Theorem 1.3.13]{zhuo}, there exists the sequence $\{\mathcal{U}_{\alpha}\}$ of
finite rank operators on $\ell^2$
such that $\mathcal{U}_{\alpha}\rightarrow\mathcal{U}$ with respect to the operator norm in $\mathcal{B}(\ell^2)$. Now
\begin{equation*}
\begin{split}
\|\mathbf{M}_{\mathcal{U}_{\alpha}} - \mathbf{M}_{\mathcal{U}}\|_{op}
&=\|\mathcal{D}_{\{g_{k}\}}o~\mathcal{U}_{\alpha}~o~\mathcal{C}_{\{f_{k}\}}
- \mathcal{D}_{\{g_{k}\}}o~\mathcal{U}~o~\mathcal{C}_{\{f_{k}\}}\|_{op}\\
& = \|\mathcal{D}_{\{g_{k}\}}o(\mathcal{U}_{\alpha}-\mathcal{U})~o~\mathcal{C}_{\{f_{k}\}}\|_{op}\\
&\leq \|\mathcal{D}_{\{g_{k}\}}\|_{op}\|\mathcal{U}_{\alpha}-\mathcal{U}\|_{op}\|\mathcal{C}_{\{f_{k}\}}\|_{op}\\
& \leq \sqrt{BB'}\|\mathcal{U}_{\alpha}-\mathcal{U}\|_{op}.\\
\end{split}
\end{equation*}
As proved, $\mathbf{M}_{\mathcal{U}_{\alpha}}$'s are finite rank operators and $\mathbf{M}_{\mathcal{U}}$ is a limit of such operators, so
$\mathbf{M}_{\mathcal{U}}$ is compact.\\
(4) Since $\mathcal{S}_p(\mathcal{H}_1, \mathcal{H}_2)$ is a two sided ideal of $\mathcal{B}(\mathcal{H}_1, \mathcal{H}_2)$, so $\mathbf{M}\in\mathcal{S}_p(\mathcal{H}_1, \mathcal{H}_2)$. Moreover, by the inequality
\begin{equation*}
\|\mathcal{D}_{\{g_{k}\}}~o~\mathcal{U}~o~\mathcal{C}_{\{f_{k}\}}\|_{p}\leq\|\mathcal{D}_{\{g_{k}\}}\|_{op}
\|\mathcal{U}\|_{p}\|\mathcal{C}_{\{f_{k}\}}\|_{op},
\end{equation*}
the result holds.\\
(5) Since $\mathcal{U}$ is positive, $\langle c, \mathcal{U}c\rangle\geq 0$ for all $c = \{c_k\}\in \ell^2(\mathbb{N})$. Hence for each $f\in \mathcal{H}$,
$$\langle f, \mathbf{M}_{\mathcal{U},\{f_k\}}(f)\rangle = \langle f, \mathcal{D}_{\{f_k\}}\mathcal{U}~\mathcal{C}_{\{f_k\}}(f) \rangle = \langle \mathcal{C}_{\{f_k\}}(f), \mathcal{U}~\mathcal{C}_{\{f_k\}}(f)\rangle\geq 0.$$
So it follows that $\mathbf{M}_{\mathcal{U},\{f_k\}}$ is positive.
\end{proof}
Here is an example which shows that a generalized Bessel multiplier may not be a compact operator
if $\mathcal{U}\notin \mathcal{K}(\ell^2)$.
\begin{example}
Let $\{f_{k}\}_{k=1}^{\infty}$ and $\{g_{k}\}_{k=1}^{\infty}$ be two Bessel sequences in $\mathcal{H}$.
Moreover let $\mathcal{U}:\ell^2(\mathbb{Z})\rightarrow\ell^2(\mathbb{Z})$ be the convolution operator which is defined as:
$$\mathcal{U}:\ell^2\rightarrow\ell^2, \hspace{3mm}\mathcal{U}(a) = (a\ast c)_j = \sum_k a_k~c_{j-k},\hspace{3mm}(0\neq c\in \ell^1).$$
We know that $\mathcal{U}$ is not compact, indeed under the Fourier transform, this operator is unitarily equivalent to the multiplication operator
$M_{\varphi}(f)$ on $L^2(\mathbb{T}), (\mathbb{T} = \text{the unit circle})$. We can see in \cite{J.B. Conway}, the multiplication operator $M_{\varphi}$
is compact only if $\varphi = 0$, now $\varphi$ is the same as the Fourier transform of $c$, so $c=0$.\\
Now, define the generalized Bessel multiplier $\mathbf{M}_\mathcal{\mathcal{U}}:\mathcal{H}\rightarrow\mathcal{H}$ by
$$\mathbf{M}_\mathcal{U} = \mathcal{D}_{\{g_{k}\}}~\mathcal{U}~\mathcal{C}_{\{f_{k}\}}.$$
Clearly
$$\mathbf{M}_\mathcal{U}(f) = \sum_j\sum_k c_{j-k}\langle f, f_k\rangle g_j = \sum_j\sum_k c_{j-k}~~ g_j\otimes f_k (f).$$
One can check that, in general, this operator is not compact.
For instance, if $\{f_{k}\}_{k=1}^{\infty}$ and $\{g_{k}\}_{k=1}^{\infty}$ are Riesz bases, then $\mathbf{M}_\mathcal{U}$ is compact
if and only if $\mathcal{U}$ is a compact operator.
\end{example}
In the next example, we give two compact generalized multipliers associated to compact symbols.
\begin{example}
Suppose that $\{a_{ij}\}$ is a complex multi-index sequence such that
$\sum_{ij}|a_{ij}|^2<\infty$. Define the operator $\mathcal{U}:\ell^2(\mathbb{N})\rightarrow \ell^2(\mathbb{N})$ by
$$\{x_{n}\}_{n=1}^{\infty}\overset{\mathcal{U}}{\mapsto} \{y_{n}\}_{n=1}^{\infty}$$ where
$y_n = \sum_{k=1}^{\infty}a_{nk}x_{k}$.
In fact, if we consider $A=[a_{ij}]$, then $A$ is the Frobenius matrix (see \cite[Definition 3.2]{Balazs-2008}) such that $\mathcal{U}x = Ax$ for $x\in \ell^2(\mathbb{N})$. Since the Frobenius matrices correspond to Hilbert-Schmidt operators, hence $\mathcal{U}$ is a Hilbert-Schmidt operator.
Now, if $\{f_{k}\}_{k=1}^{\infty}$ and $\{g_{k}\}_{k=1}^{\infty}$ are two Bessel sequences for  Hilbert spaces $\mathcal{H}_1$ and $\mathcal{H}_2$, respectively, one can define the corresponding  generalized Bessel multiplier
$\mathbf{M}:\mathcal{H}_1\rightarrow\mathcal{H}_2$ by
$$
f\mapsto\sum_{n=1}^{\infty}\sum_{k=1}^{\infty}a_{nk}
(g_n\otimes f_k)(f).
$$
Then, by \cite[Proposition 3.6]{Balazs-2008}, $\mathbf{M}$ is a Hilbert-Schmidt operator. Now, if
$\mathcal{U}$ is the operator corresponding to
the following tri-block diagonal matrix
\begin{equation*}
A = [a_{ij}] = \left(
\begin{array}{ccccccc}
1 & 1 & 0 & 0 & 0 & 0 & \ldots\\
1 & \frac{1}{\sqrt2} & \frac{1}{\sqrt2} & 0 & \ddots & \ddots & \ddots\\
0 & \frac{1}{\sqrt2} & \frac{1}{\sqrt3} &\frac{1}{\sqrt3} & \ddots & \ddots & \ddots\\
0 & 0 & \frac{1}{\sqrt3} & \ddots & \ddots & \ddots &\ddots\\
\vdots & \ddots & \ddots & \ddots & \ddots & \ddots & \ddots\\
\vdots & \ddots & \ddots & \ddots & \ddots & \frac{1}{\sqrt n} & \ddots\\
\vdots & \ddots & \ddots & \ddots & \ddots & \ddots & \ddots
\end{array} \right),
\end{equation*}
then, by \cite[Theorem 2]{Bakic-1999}, $\mathcal{U}$ defines a compact operator on $\ell^2(\mathbb{N})$
and so by Proposition \ref{main theorem}, the generalized Bessel multiplier $\mathbf{M}_{\mathcal{U}}:\mathcal{H}_1\rightarrow \mathcal{H}_2$ is compact.
\end{example}
The following two propositions, which are similar to \cite[Proposition 3.9]{Khosravi-Mirzaee Azandaryani}, indicate that the underlying Bessel sequences
of bounded below generalized multipliers must be frames.
\begin{prop}\label{bounded below}
Let $\{f_{k}\}_{k=1}^{\infty}$ and $\{g_{k}\}_{k=1}^{\infty}$ be Bessel sequences for $\mathcal{H}$.
\begin{enumerate}
  \item If $\mathbf{M}_{\mathcal{U},\{g_k\},\{f_k\}}$ is
  a bounded below operator on $\mathcal{H}$, then $\{f_{k}\}_{k=1}^{\infty}$ is a frame for $\mathcal{H}$.
  \item If there exists $A>0$ such that for each $f\in\mathcal{H}$, $A\|f\|^2\leq|\langle \mathbf{M}_{\mathcal{U},\{g_k\},\{f_k\}}f, f \rangle|$,
  then $\{f_{k}\}_{k=1}^{\infty}$ and $\{g_{k}\}_{k=1}^{\infty}$ are both frames.
\end{enumerate}
\end{prop}
\begin{proof}
(1) Since $\mathbf{M}_{\mathcal{U},\{g_k\},\{f_k\}}$ is bounded below, there exists a constant $C>0$ such that
$C\|f\|\leq \|\mathbf{M}_{\mathcal{U},\{g_k\},\{f_k\}}(f)\|$ for every $f\in \mathcal{H}$.
Now, for each non-zero element $f\in\mathcal{H}$, one can choose some $h\in\mathcal{H}$ with $\|h\|=1$ and $C\|f\|\leq |\langle \mathbf{M}_{\mathcal{U},\{g_k\},\{f_k\}}(f), h\rangle|$.
Thus, if $B$ is the Bessel bound for $\{g_{k}\}_{k=1}^{\infty}$, we have
\begin{equation*}
\begin{split}
C\|f\|\leq|\langle\mathbf{M}_{\mathcal{U},\{g_k\},\{f_k\}}(f), h\rangle|
& = |\langle \mathcal{D}_{\{g_{k}\}}~o~\mathcal{U}~o~\mathcal{C}_{\{f_{k}\}}(f), h\rangle|\\
& = |\langle \mathcal{C}_{\{f_{k}\}}(f), \mathcal{U}^{*}o~\mathcal{C}_{\{g_{k}\}}(h)\rangle|\\
&\leq \|\mathcal{C}_{\{f_{k}\}}(f)\|_{2} ~\|\mathcal{U}^{*}o~\mathcal{C}_{\{g_{k}\}}(h)\|_{2}\\
& \leq\sqrt{B}\|\mathcal{U}\|_{op} \Big(\sum_{k=1}^{\infty}|\langle f, f_k\rangle|^{2} \Big)^{1/2},\\
\end{split}
\end{equation*}
hence
$(\dfrac{C^2}{B\|\mathcal{U}\|_{op}^{2}})\|f\|^2\leq\sum_{k=1}^{\infty}|\langle f, f_k\rangle|^{2}$. The upper bound is obvious, since $\{f_{k}\}_{k=1}^{\infty}$ is
a Bessel sequence. Therefore $\{f_{k}\}_{k=1}^{\infty}$ is a frame.\\
(2) For each $f\in\mathcal{H}$, since $A\|f\|^2\leq|\langle \mathbf{M}_{\mathcal{U},\{g_k\},\{f_k\}}f, f \rangle|\leq\|\mathbf{M}_{\mathcal{U},\{g_k\},\{f_k\}}f\| \|f\|$, it follows that $\mathbf{M}_{\mathcal{U},\{g_k\},\{f_k\}}$ is bounded below and so by part (1), $\{f_k\}$ is a frame. Furthermore, by the same
argument we have $A\|f\|^2\leq\|\mathbf{M}_{\mathcal{U^{*}},\{f_k\},\{g_k\}}f\|\|f\|$. Hence, $\mathbf{M}_{\mathcal{U^{*}},\{f_k\},\{g_k\}}$ is
also a bounded below operator and so $\{g_k\}$ is a frame.
\end{proof}
\begin{prop}
Let $\{f_{k}\}_{k=1}^{\infty}$ and $\{g_{k}\}_{k=1}^{\infty}$ be Bessel sequences for $\mathcal{H}$.
\begin{enumerate}
  \item If there exist $\lambda_{1}<1$ and $\lambda_{2}>-1$ such that
  $$\|f - \mathbf{M}_{\mathcal{U},\{g_k\},\{f_k\}}(f)\|\leq \lambda_{1}\|f\|+\lambda_{2}\|\mathbf{M}_{\mathcal{U},\{g_k\},\{f_k\}}(f)\|,\hspace*{4mm} (f\in\mathcal{H}),$$
  then $\{f_{k}\}_{k=1}^{\infty}$ is a frame for $\mathcal{H}$.
  \item If there exists $\lambda\in[0, 1)$ such that $\|f - \mathbf{M}_{\mathcal{U},\{g_k\},\{f_k\}}(f)\|\leq \lambda\|f\|$, for every $f\in\mathcal{H}$,
  then $\{f_{k}\}_{k=1}^{\infty}$ and $\{g_{k}\}_{k=1}^{\infty}$ are frames for $\mathcal{H}$.
\end{enumerate}
\end{prop}
\begin{proof}
(1) Since
$\|f -\mathbf{M}_{\mathcal{U},\{g_k\},\{f_k\}}(f)\|\geq \|f\| - \|\mathbf{M}_{\mathcal{U},\{g_k\},\{f_k\}}(f)\|$, so by assumption
$$\lambda_1\|f\|+\lambda_2\|\mathbf{M}_{\mathcal{U},\{g_k\},\{f_k\}}(f)\|\geq\|f\| - \|\mathbf{M}_{\mathcal{U},\{g_k\},\{f_k\}}(f)\|,$$
hence $\|\mathbf{M}_{\mathcal{U},\{g_k\},\{f_k\}}(f)\|\geq\dfrac{1-\lambda_{1}}{1+\lambda_{2}}\|f\|$. Now the result follows from Proposition \ref{bounded below}.\\
(2) We have
\begin{equation*}
\begin{split}
\|f- \mathbf{M}_{\mathcal{U^{*}},\{f_k\},\{g_k\}}(f)\|
& =\|(Id_{\mathcal{H}}-\mathbf{M}_{\mathcal{U},\{g_k\},\{f_k\}})^{*}(f)\|\\
& \leq \|Id_{\mathcal{H}}-\mathbf{M}_{\mathcal{U},\{g_k\},\{f_k\}}\|~\|f\|\\
& \leq\lambda\|f\|,\\
\end{split}
\end{equation*}

therefore by using part (1), we obtain that $\{f_{k}\}_{k=1}^{\infty}$ and $\{g_{k}\}_{k=1}^{\infty}$ are frames.
\end{proof}
In \cite{Balazs-2008}, Balazs established Bessel sequences, frames and Riesz bases of the class of Hilbert-Schmidt operators
using tensor products of the same sequences in the associated Hilbert spaces.
Indeed, if $\{g_k\}$ and $\{f_i\}$ are frames (Riesz bases) with bounds $A, B$ and $A', B'$ for Hilbert spaces $\mathcal{H}_2$ and $\mathcal{H}_1$,
respectively, then $\{f_i\otimes g_k\}$ is a frame (Riesz basis) for $\mathcal{S}_2(\mathcal{H}_2, \mathcal{H}_1)$ with bounds $AA'$ and $BB'$.
Moreover, the synthesis operator associated with $\{f_i\otimes g_k\}$ is the generalized multiplier $\mathbf{M}_{\mathcal{U},\{g_k\},\{f_i\}}$, where
$\mathcal{U}\in \mathcal{S}_2(\ell^2)$. So by \eqref{eqn:Riesz intro}, for Riesz bases $\{g_k\}$ and $\{f_i\}$, we have
$$\sqrt{AA'}\|\mathcal{U}\|_{\mathcal{S}_2(\ell^2)}\leq\|\mathbf{M}_{\mathcal{U},\{g_k\},\{f_i\}}\|_{\mathcal{S}_2(\mathcal{H}_2, \mathcal{H}_1)}\leq\sqrt{BB'}\|\mathcal{U}\|_{\mathcal{S}_2(\ell^2)}.$$
The following proposition gives a lower bound for generalized multipliers associated to Riesz bases.
Recall that two sequences $\{f_k\}_{k=1}^{\infty}$ and $\{g_k\}_{k=1}^{\infty}$ in a Hilbert space
are biorthogonal if
\begin{equation}\label{biorthogonal condition}
\langle f_k,g_j\rangle =\delta_{k,j}.
\end{equation}
For any Riesz basis $\{f_k\}_{k=1}^{\infty}$, there is a unique biorthogonal sequence, which is also a Riesz basis and denoted by $\{\tilde{f}_k\}_{k=1}^{\infty}$.
\begin{prop}
Let $\{f_{k}\}_{k=1}^{\infty}$ and $\{g_{k}\}_{k=1}^{\infty}$ be Riesz bases with bounds, respectively, $A,B$ and $A',B'$. Then
\begin{equation*}
K\sqrt{AA'}\leq\|\mathbf{M}_{\mathcal{U},\{g_k\},\{f_k\}}\|_{op}\leq\sqrt{BB'}\|\mathcal{U}\|_{op},
\end{equation*}
where $K=\sup\{\|\mathcal{U}(e_n)\|_{\ell^2}; n\in\mathbb{N}\}$ and $\{e_n\}$ is the canonical basis of $\ell^2$.
\end{prop}
\begin{proof}
As we have seen in Theorem \ref{main theorem}, the upper bound condition is satisfied. Now for the lower bound, consider an arbitrary element $k_0$ in the index set. Then
\begin{equation*}
\begin{split}
\mathbf{M}_{\mathcal{U},\{g_k\},\{f_k\}}(\tilde{f}_{k_0})
& =\mathcal{D}_{\{g_{k}\}}o~\mathcal{U}~o~\mathcal{C}_{\{f_{k}\}}(\tilde{f}_{k_0})\\
& =\mathcal{D}_{\{g_{k}\}}o~\mathcal{U}\{\langle\tilde{f}_{k_0},f_k\rangle_k\}\\
& =\mathcal{D}_{\{g_{k}\}}o~\mathcal{U}\{\delta_{k_{0},k}\}_{k},
\end{split}
\end{equation*}
therefore
\begin{equation*}
\begin{split}
\|\mathbf{M}_{\mathcal{U},\{g_k\},\{f_k\}}\|_{op}
= &\sup\{\dfrac{\|\mathbf{M}_{\mathcal{U},\{g_k\},\{f_k\}}(f)\|_{\mathcal{H}}}{\| f\|_{\mathcal{H}}}\}\\
& \geq \dfrac{\parallel\mathbf{M}_{\mathcal{U},\{g_k\},\{f_k\}}(\tilde{f}_{k_0})\parallel_{\mathcal{H}}}
{\|\tilde{f}_{k_0}\|_{\mathcal{H}}}\\
& =\dfrac{\|\mathcal{D}_{\{g_{k}\}}~o~\mathcal{U}\{\delta_{k_{0},k}\}_{k}\|_{\mathcal{H}}}{\|\tilde{f}_{k_0}\|_{\mathcal{H}}}\\
& \geq \dfrac{\sqrt{A'}\|\mathcal{U}\{\delta_{k_{0},k}\}_{k}\|_{2}}{1/\sqrt{A}},\\
\end{split}
\end{equation*}
since $k_0$ is chosen arbitrary, it follows that $K\sqrt{AA'}\leq \|\mathbf{M}_{\mathcal{U},\{g_k\}, \{f_k\}}\|_{op}$.
\end{proof}
It is mentioned that for an orthonormal sequence $\{e_{k}\}$, the combination of multipliers is just the combination of symbols as
$\mathbf{M}_{\mathcal{U}_{1},\{e_{k}\}}o~\mathbf{M}_{\mathcal{U}_{2},\{e_{k}\}} = \mathbf{M}_{\mathcal{U}_{1}o\mathcal{U}_{2},\{e_{k}\}}$.
This is true for all biorthogonal Bessel sequences as follows.
\begin{prop}
Let $\{f_{k}\}_{k=1}^{\infty},\{g_{k}\}_{k=1}^{\infty},\{l_{k}\}_{k=1}^{\infty}\subseteq\mathcal{H}_{2}$ and $\{h_{k}\}_{k=1}^{\infty}\subseteq\mathcal{H}_{1}$ be Bessel sequences, such that $\{l_k\}$ and$\{f_k\}$ are
biorthogonal to each other. Then
\begin{equation*}
\mathbf{M}_{\mathcal{U}_1,\{g_k\},\{f_k\}}~o~\mathbf{M}_{\mathcal{U}_2,\{l_k\},\{h_k\}}
=\mathbf{M}_{\mathcal{U}_1 o\mathcal{U}_2, \{g_k\},\{h_k\}}.
\end{equation*}
\end{prop}
\begin{proof}
The details are left to the reader as a simple exercise.
\end{proof}
\section{Sufficient and necessary conditions for invertibility of generalized multipliers}
As multipliers are important for applications, it is interesting to determine their inverse. Also the invertibility of multipliers helps us to obtain more reconstruction formula.

Moreover, as is mentioned above, the concepts multipliers and reproducing pairs are closely related to each other.
Reproducing pairs is recent development of
frames that generates a bounded analysis and synthesis process while the frame
condition can be omitted at both stages.
Recall that if $\{f_k\}_{k=1}^{\infty}$ and $\{g_k\}_{k=1}^{\infty}$ are two families in $\mathcal{H}$, then the pair
$(\{f_k\},\{g_k\})$ is called a \emph{reproducing pair} for $\mathcal{H}$ if the operator $S:\mathcal{H}\rightarrow\mathcal{H}$,
weakly given by
\begin{equation*}\label{reproducing pair}
\langle Sf, g\rangle = \sum_k\langle f, f_k\rangle\langle g_k, g\rangle,
\end{equation*}
is bounded and boundedly invertible.
Clearly, if the multiplier $\mathbf{M}_{\mathcal{I},\{g_{k}\},\{f_{k}\}}$ is invertible, then the pair $(\{f_k\},\{g_k\})$
forms a reproducing pair for $\mathcal{H}$.

Some conditions of the possibility of invertibility of generalized multipliers have been investigated more recently in work of
Balazs and Rieckh (see \cite{Balazsss-2016}). They obtained a representation of the inverse of these operators using the pseudo-inverse operator $\mathcal{U}^{\dagger}$.
In this section, the invertibility of generalized multipliers is studied. We provide some conditions for invertibility
of such operators depending on the underlying symbol and sequences.

It was shown, In \cite[Theorem 3.5]{Balazs-2008}, that if $\{f_{k}\}_{k=1}^{\infty}$ and $\{g_{k}\}_{k=1}^{\infty}$
be two Riesz bases of $\mathcal{H}$,
then the mapping $\mathcal{U}\mapsto \mathbf{M}_{\mathcal{U},\{g_{k}\},\{f_{k}\}}$ from $\mathcal{B}(\ell^2)$ to $\mathcal{B}(\mathcal{H})$ is invertible.
The following proposition gives a different viewpoint for the invertibility of generalized Riesz multipliers. Indeed, we show that
for Riesz sequences, the generalized multiplier is invertible if and only if its symbol is. This covers the part (i) of \cite[Theorem 5.1]{Stoeva-2012}.
\begin{prop}\label{nec and suf con for inver}
Let $\mathcal{U}\in\mathcal{B}(\ell^2)$ and $\{f_{k}\}_{k=1}^{\infty}$ and $\{g_{k}\}_{k=1}^{\infty}$ be Riesz bases. Then
$\mathcal{U}$ is invertible if and only if the associated generalized multiplier $\mathbf{M}_{\mathcal{U},\{g_{k}\},\{f_{k}\}}$ is invertible.
\end{prop}
\begin{proof}
At the first, suppose that $\mathcal{U}$ is invertible. Then for each $f\in\mathcal{H}$, by using the orthogonality property of dual Riesz bases, we have

\begin{equation*}
\begin{split}
(\mathbf{M}_{\mathcal{U},\{g_{k}\},\{f_{k}\}}~o~\mathbf{M}_{\mathcal{U}^{-1},\{\tilde{f}_k\},\{\tilde{g}_k\}})(f)
& =(\mathcal{D}_{\{g_k\}}o~\mathcal{U}~o~\mathcal{C}_{\{f_k\}})~o~(\mathcal{D}_{\{\tilde{f}_k\}}o~\mathcal{U}^{-1}o~\mathcal{C}_{\{\tilde{g}_k\}})(f)\\
& =(\mathcal{D}_{\{g_k\}}o~\mathcal{U}~o~\mathcal{U}^{-1}o~\mathcal{C}_{\{\tilde{g}_k\}})(f)\\
& =(\mathcal{D}_{\{g_k\}}o~\mathcal{C}_{\{\tilde{g}_k\}})(f)
 =Id(f),\\
\end{split}
\end{equation*}
so $\mathbf{M}^{-1}_{\mathcal{U},\{g_{k}\},\{f_{k}\}}=\mathbf{M}_{\mathcal{U}^{-1},\{\tilde{f}_k\},\{\tilde{g}_k\}}$.

\vspace{3mm}
Conversely, let $\mathbf{M}_{\mathcal{U},\{g_{k}\},\{f_{k}\}}$ has the inverse $\mathbf{N}$. Then for each $c\in\ell^2$
\begin{equation*}
\begin{split}
\mathcal{U}o\Big(\mathcal{C}_{\{f_k\}}o\mathbf{N}o\mathcal{D}_{\{g_k\}}\Big)(c)
& = \Big(\mathcal{U}o~\mathcal{C}_{\{f_k\}}o\mathbf{N}\Big)o\mathcal{D}_{\{g_k\}}(c)\\
& = \mathcal{C}_{\{\tilde{g}_k\}}o\mathcal{D}_{\{g_k\}}o\Big(\mathcal{U}o~\mathcal{C}_{\{f_k\}}o\mathbf{N}\Big)o\mathcal{D}_{\{g_k\}}(c)\\
& = \mathcal{C}_{\{\tilde{g}_k\}}o\Big(\mathcal{D}_{\{g_k\}}o\mathcal{U}o~\mathcal{C}_{\{f_k\}}o\mathbf{N}\Big)o\mathcal{D}_{\{g_k\}}(c)\\
&  = \Big(\mathcal{C}_{\{\tilde{g}_k\}}o~\mathcal{D}_{\{g_k\}}\Big)(c)\\
& =Id(c),\\
\end{split}
\end{equation*}
and by the similar argument,
$\Big(\mathcal{C}_{\{f_k\}}o\mathbf{N}o\mathcal{D}_{\{g_k\}}\Big)o~\mathcal{U}= Id$ and therefore $\mathcal{U}$ is invertible.
\end{proof}

In \cite[Proposition 3.1]{Stoeva-2012}, it is proved that if one of the sequences is Bessel, invertibility of $\mathbf{M}_{(1),\{g_{k}\},\{f_{k}\}}$ implies that
the other one must satisfies the lower frame condition. The following proposition gives a general version of this result.
\begin{prop}\label{lower frame condition}
Let $\mathbf{M}_{\mathcal{U},\{g_k\},\{f_k\}}$ be an invertible operator on $\mathcal{H}$.
If $\{f_{k}\}_{k=1}^{\infty}$ $(\{g_{k}\}_{k=1}^{\infty})$ is a Bessel sequence for $\mathcal{H}$ with bound $B$,
then $\{g_{k}\}_{k=1}^{\infty}$ $(\{f_{k}\}_{k=1}^{\infty})$ satisfies the lower frame condition for $\mathcal{H}$.
\end{prop}
\begin{proof}
For every $f, g\in\mathcal{H}$
\begin{equation*}
\begin{split}
|\langle \mathbf{M}_{\mathcal{U},\{g_k\},\{f_k\}}(f), g\rangle|
& = |\sum_k (\mathcal{U}\langle f,f_k\rangle)\langle g_k, g\rangle|\\
& \leq \Big(\sum_k |(\mathcal{U}\langle f,f_k\rangle)|^2\Big)^{1/2} \Big(\sum_k|\langle g_k, g\rangle|^2\Big)^{1/2}\\
& \leq \|\mathcal{U}\| \|f\|\sqrt{B} \Big(\sum_k|\langle g_k, g\rangle|^2\Big)^{1/2}.\\
\end{split}
\end{equation*}
Now for $f=\mathbf{M}^{-1}_{\mathcal{U},\{g_k\},\{f_k\}}(g)$, we have
$$\|g\|\leq \|\mathcal{U}\|\|\mathbf{M}^{-1}_{\mathcal{U},\{g_k\},\{f_k\}}\|\sqrt{B}\Big(\sum_k|\langle g_k, g\rangle|^2\Big)^{1/2}$$
and so the result holds.
\end{proof}
We immediately obtain the following useful corollary.
\begin{cor}
Suppose $\{g_{k}\}_{k=1}^{\infty}$ is a Riesz basis and $\{f_{k}\}_{k=1}^{\infty}$ is a Bessel sequence for a Hilbert space $\mathcal{H}$. Moreover, let
$\mathcal{U}\in\mathcal{B}(\ell^2)$ be a bijective operator. Then $\mathbf{M}_{\mathcal{U},\{g_{k}\},\{f_{k}\}}$ is invertible if and only if
$\{f_{k}\}_{k=1}^{\infty}$ is a Riesz basis.
\end{cor}
\begin{proof}
If $\{f_{k}\}_{k=1}^{\infty}$ is a Riesz basis, it is obvious, by Proposition \ref{nec and suf con for inver}, that $\mathbf{M}_{\mathcal{U},\{g_{k}\},\{f_{k}\}}$ is invertible.
Conversely, suppose $\mathbf{M}_{\mathcal{U},\{g_{k}\},\{f_{k}\}}$ is invertible. Then by Proposition \ref{lower frame condition}, $\{f_{k}\}_{k=1}^{\infty}$ is a frame for
$\mathcal{H}$. Furthermore, invertibility of $\mathbf{M}_{\mathcal{U},\{g_{k}\},\{f_{k}\}} = \mathcal{D}_{\{g_k\}}o\mathcal{U}o~\mathcal{C}_{\{f_k\}}$ implies that the analysis operator $\mathcal{C}_{\{f_k\}}$ is surjective. The result now follows by \cite[Theorem 7.1.1]{Christensen-2016}.
\end{proof}
Our next result is analogous to \cite[Proposition 4.2]{Stoeva-2012} which gives
sufficient conditions for the invertibility of generalized frame multipliers.
\begin{prop}
Suppose that $\{f_{k}\}_{k=1}^{\infty}$ is a frame for $\mathcal{H}$ with bounds $A,B$ and $\{g_{k}\}_{k=1}^{\infty}$ is an arbitrary sequence in $\mathcal{H}$.
Let $0<\mu<\dfrac{1}{B}\Big(\dfrac{AB^2-A^2B}{A^2+B^2}\Big)^2$ be such that for each $f\in \mathcal{H}$,
$$\sum_k |\langle f, f_k - g_k\rangle|^2\leq \mu \|f\|^2.$$
Moreover let $\|\mathcal{U}-Id\|< \dfrac{A^2}{B^2}$, where $\mathcal{U}$ is a non-zero bounded linear operator on $\ell^2$.
Then the following holds.
\begin{enumerate}
  \item $\{g_{k}\}_{k=1}^{\infty}$ is a frame for $\mathcal{H}$.
  \item $\mathbf{M}_{\mathcal{U},\{f_k\},\{g_k\}}$ is invertible.
\end{enumerate}
\end{prop}
\begin{proof}
(1) It follows from \cite[Corollary 22.1.5]{Christensen-2016}.\\
(2) Let $\mathcal{S}$ be the frame operator associated with $\{f_{k}\}_{k=1}^{\infty}$. Note first that for every $f\in\mathcal{H}$,
\begin{equation*}
\begin{split}
\|\mathbf{M}_{\mathcal{U},\{f_k\},\{f_k\}}(f)-\mathcal{S}(f)\|
& = \|\mathbf{M}_{\mathcal{U},\{f_k\},\{f_k\}}(f)-\mathbf{M}_{Id,\{f_k\},\{f_k\}}(f)\|\\
& = \|\mathbf{M}_{\mathcal{U}-Id,\{f_k\},\{f_k\}}(f)\|\\
& \leq \|\mathcal{U}-Id\|\|\mathcal{D}_{\{f_k\}}\|\|\mathcal{C}_{\{f_k\}}\|\|f\|\\
& < (A^2/B)\|f\|.\\
\end{split}
\end{equation*}
Since $\dfrac{A^{2}}{B}\leq\dfrac{1}{\|\mathcal{S}^{-1}\|}$, we conclude, by \cite[Proposition 2.2]{Stoeva-2012}, that $\mathbf{M}_{\mathcal{U},\{f_k\},\{f_k\}}$ is invertible.
Now
\begin{equation*}
\begin{split}
\|\mathbf{M}_{\mathcal{U},\{f_k\},\{g_k\}}(f)-\mathbf{M}_{\mathcal{U},\{f_k\},\{f_k\}}(f)\|
& = \|\mathbf{M}_{\mathcal{U},\{f_k\},\{g_k\}-\{f_k\}}(f)\|\\
& \leq \|\mathcal{U}\| \|\mathcal{D}_{\{f_k\}}\|\|\mathcal{C}_{\{g_k\}-\{f_k\}}(f)\|\\
& \leq \|\mathcal{U}\|\sqrt{B}\sqrt{\mu}\|f\|.\\
\end{split}
\end{equation*}
To complete the proof, it is enough to check that $\|\mathcal{U}\|\sqrt{B}\sqrt{\mu} <\dfrac{1}{\|\mathbf{M}^{-1}_{\mathcal{U},\{f_k\},\{f_k\}}\|}$. But
\begin{equation*}
\|\mathcal{U}\|\sqrt{B}\sqrt{\mu}<\dfrac{A^2+B^2}{B^2}\sqrt{B}\sqrt{\mu}<A-\dfrac{A^2}{B}\leq \dfrac{1}{\|S^{-1}\|}-\dfrac{A^2}{B}\leq
\dfrac{1}{\|\mathbf{M}^{-1}_{\mathcal{U},\{f_k\},\{f_k\}}\|},
\end{equation*}
where the last inequality follows from the second part of \cite[Proposition 2.2]{Stoeva-2012}.
\end{proof}
\section{Perturbation of generalized Multipliers}
A generalized Bessel multiplier clearly depends on the chosen symbol, analysis and synthesis sequence. We check some states of convergence. Before
presenting our results, we borrow the following
essential definition from \cite{Balazs-2007}.
\begin{dfn}
Let $\{f_k\}_{k}$ and $\{f_{k}^{(l)}\}_{k}$ be sequences of elements of $\mathcal{H}$, for all $l\in\mathbb{N}$. The sequences $\{f_{k}^{(l)}\}_{k}$
are said to converges to $\{f_k\}_{k}$ in an $\ell^p$ sense, denoted by $\{f_{k}^{(l)}\}\overset{\parallel .\parallel_{\ell^p}}{\longrightarrow}\{f_{k}\}$,
if for all $\epsilon>0$, there exists $N>0$ such that
$\big(\sum_{k}\|{f_{k}^{(l)}}-{f_{k}}\|_{\mathcal{H}}^{p}\big)^{\frac{1}{p}}<\epsilon$ for all $l\geq N$.
\end{dfn}
\begin{prop}
Let $\mathcal{U}\in{\mathcal{B}(\ell^2)}$ and $\mathbf{M}_{\mathcal{U},\{g_{k}\},\{f_{k}\}}$ be the generalized Bessel multiplier
for the Bessel sequences $\{f_{k}\}_{k=1}^{\infty}$ and $\{g_{k}\}_{k=1}^{\infty}$ with Bessel bounds $B$ and $B'$, respectively. Then for the indexed sequences $\{\mathcal{U}^{(l)}\}$ in $\mathcal{B}(\ell^2)$ and Bessel sequences $\{f_{k}^{(l)}\}$ and $\{g_{k}^{(l)}\}$ we have:
\begin{enumerate}
  \item
 (a) If $\mathcal{U}^{(l)}\overset{\parallel .\parallel_{op}}{\longrightarrow}\mathcal{U}$, then $\parallel\mathbf{M}_{\mathcal{U}^{(l)},\{g_{k}\},\{f_{k}\}}-\mathbf{M}_{\mathcal{U},\{g_{k}\},\{f_{k}\}}\parallel_{op}\rightarrow 0$.\\
 (b) If $\mathcal{U}^{(l)}\overset{\parallel .\parallel_{1}}{\longrightarrow}\mathcal{U}$, then $\parallel\mathbf{M}_{\mathcal{U}^{(l)},\{g_{k}\},\{f_{k}\}}-\mathbf{M}_{\mathcal{U},\{g_{k}\},\{f_{k}\}}\parallel_{1}\rightarrow 0$.\\
 (c) If $\mathcal{U}^{(l)}\overset{\parallel .\parallel_{2}}{\longrightarrow}\mathcal{U}$, then $\parallel\mathbf{M}_{\mathcal{U}^{(l)},\{g_{k}\},\{f_{k}\}}-\mathbf{M}_{\mathcal{U},\{g_{k}\},\{f_{k}\}}\parallel_{2}\rightarrow 0$.\\
   \item
   (a)If $\{g_{k}^{(l)}\}\overset{\parallel .\parallel_{\ell^1}}{\longrightarrow}\{g_{k}\}$, then
   $\parallel\mathbf{M}_{\mathcal{U},\{g_{k}^{(l)}\},\{f_{k}\}}-\mathbf{M}_{\mathcal{U},\{g_{k}\},\{f_{k}\}}\parallel_{1}\rightarrow 0$.\\
   (b)If $\{g_{k}^{(l)}\}\overset{\parallel .\parallel_{\ell^2}}{\longrightarrow}\{g_{k}\}$, then
   $\parallel\mathbf{M}_{\mathcal{U},\{g_{k}^{(l)}\},\{f_{k}\}}-\mathbf{M}_{\mathcal{U},\{g_{k}\},\{f_{k}\}}\parallel_{2}\rightarrow 0$.
   \item For Bessel sequence $\{f_{k}^{(l)}\}$ converging to $\{f_{k}\}$, similar properties as in (2) satisfy.
\end{enumerate}
\end{prop}
\begin{proof}
(1) (b)This is a direct result of Theorem \ref{main theorem} as follows:
\begin{equation*}
\begin{split}
\parallel\mathbf{M}_{\mathcal{U}^{(l)},\{g_{k}\},\{f_{k}\}}-\mathbf{M}_{\mathcal{U},\{g_{k}\},\{f_{k}\}}\parallel_{1}
& =\parallel\mathbf{M}_{(\mathcal{U}^{(l)}-\mathcal{U}),\{g_{k}\},\{f_{k}\}}\parallel_{1}\\
& =\parallel\mathcal{D}_{\{g_k\}}o~(\mathcal{U}^{(l)}-\mathcal{U})~o~\mathcal{C}_{\{f_k\}}\parallel_{1}\\
& \leq\parallel\mathcal{D}\parallel_{op}~\parallel\mathcal{U}^{(l)}-\mathcal{U}\parallel_{1}~\parallel\mathcal{C}\parallel_{op}\\
& \leq\sqrt{BB'}\epsilon.
\end{split}
\end{equation*}
(a) and (c) can be proved similarly.\\
(2)
(a) By the polar decomposition theorem for synthesis operator $\mathcal{D}$, there exists partial isometry $\mathcal{V}$ such that $|\mathcal{D}|=\mathcal{V}\mathcal{D}$. So
\begin{equation*}
\begin{split}
\|\mathcal{D}_{\{g_{k}^{(l)}\}}-\mathcal{D}_{\{g_{k}\}}\|_{1}
& =\sum_{n=1}^{\infty}\langle |\mathcal{D}_{\{{g_{k}^{(l)}}-{g_{k}\}}}|e_n,e_n\rangle\\
& =\sum_{n=1}^{\infty}\langle (\mathcal{D}_{\{{g_{k}^{(l)}}-{g_{k}\}}})e_n,\mathcal{V}e_n\rangle\\
& =\sum_{n=1}^{\infty}\langle {g_{n}^{(l)}}-{g_{n}},\mathcal{V}e_n\rangle\\
& \leq\sum_{n=1}^{\infty}|\langle {g_{n}^{(l)}}-{g_{n}},\mathcal{V}e_n\rangle |\\
& \leq \|\mathcal{V}\|\sum_{n=1}^{\infty}|{g_{n}^{(l)}}-{g_{n}}|\leq\epsilon,\\
\end{split}
\end{equation*}
hence
\begin{equation*}
\begin{split}
\|\mathbf{M}_{\mathcal{U},\{g_{k}^{(l)}\},\{f_{k}\}}-\mathbf{M}_{\mathcal{U},\{g_{k}\},\{f_{k}\}}\|_{1}
& =\|(\mathcal{D}_{\{g_{k}^{(l)}\}}-\mathcal{D}_{\{g_{k}\}})~o~\mathcal{U}~o~\mathcal{C}_{\{f_k\}}\|_{1}\\
&\leq\|\mathcal{D}_{\{g_{k}^{(l)}\}}-\mathcal{D}_{\{g_{k}\}}\|_{1}~\|\mathcal{U}\|_{op}~\|\mathcal{C}_{\{f_k\}}\|_{op}\\
&\leq\|\mathcal{U}\|_{op}\sqrt{B}\epsilon.
\end{split}
\end{equation*}
(b)
\begin{equation*}
\|\mathcal{D}_{\{g_{k}^{(l)}\}}-\mathcal{D}_{\{g_{k}\}}\|_{2}
=(\sum_{n=1}^{\infty}\|(\mathcal{D}_{\{ g_{k}^{(l)}-{g_k}\}}) e_n\|^2)^{1/2}
=(\sum_{n=1}^{\infty}\| g_{n}^{(l)}-{g_n}\|^2)^{1/2}
\leq\epsilon,
\end{equation*}
hence
\begin{equation*}
\begin{split}
\|\mathbf{M}_{\mathcal{U},\{g_{k}^{(l)}\},\{f_{k}\}}-\mathbf{M}_{\mathcal{U},\{g_{k}\},\{f_{k}\}}\|_{2}
& =\|(\mathcal{D}_{\{g_{k}^{(l)}\}}-\mathcal{D}_{\{g_{k}\}})~o~\mathcal{U}~o~\mathcal{C}_{\{f_k\}}\|_{2}\\
&\leq\|\mathcal{D}_{\{g_{k}^{(l)}\}}-\mathcal{D}_{\{g_{k}\}}\|_{2}~\|\mathcal{U}\|_{op}~\|\mathcal{C}_{\{f_k\}}\|_{op}\\
&\leq\|\mathcal{U}\|_{op}\sqrt{B}\epsilon.
\end{split}
\end{equation*}
(3) Use a corresponding argumentation as in (2).
\end{proof}
\section{Conclusion}
In this paper, we discussed the generalized multipliers, which is the composition of analysis operator, an operator $\mathcal{U}\in\mathcal{B}(\ell^2)$ and
synthesis operator. At first, we have started with collecting some important properties of these operators. We found that when the symbol belongs to certain
operator spaces, such as bonded, positive, compact or Hilbert-Schmidt operators, the associated generalized Bessel multiplier belongs to the same
operator spaces, in the special sense. Also, we have given the lower and upper bounds for the generalized Riesz multipliers.
In the sequel, we were interested in the investigation of the invertibility of generalized multipliers. As another result, we showed
that the generalized multiplier depends continuously on the symbol and on the involved sequences. Moreover, we gave some examples to support our results.
\section{Acknowledgement}
The authors would like to thank the referees for their valuable
comments and suggestions which improved the manuscript.

\end{document}